\documentclass[
final
]{dmtcs-episciences}

\usepackage[utf8]{inputenc}
\usepackage{subfigure}
\usepackage{enumerate}

\newtheorem{theorem}{Theorem}
\newtheorem{lemma}[theorem]{Lemma}
\newtheorem{corollary}[theorem]{Corollary}

\newtheorem{conjecture}{Conjecture}

\usepackage[square, comma, sort&compress, numbers]{natbib}

\author[Yiqiao Wang et. al]{Yiqiao Wang\affiliationmark{1}\thanks{Research supported partially by NSFC (Nos.\,12071048; 12161141006)}
  \and Ning Song\affiliationmark{2}
  \and Jianfeng Wang\affiliationmark{2}\thanks{Research supported partially by NSFC (Nos.\,11971274)}
  \and Weifan Wang\affiliationmark{2}\thanks{Research supported partially by NSFC (Nos.\,12031018; 12226303); Corresponding author. Email: wwf@zjnu.cn}}
\title[The strong  chromatic index of $1$-planar graphs]{The strong chromatic index of $1$-planar graphs}
\affiliation{
  Faculty of Science, Beijing University of Technology, Beijing, China\\
  School of Mathematics and Statistics, Shandong University of Technology, Zibo, China}
\begin{document}
\publicationdata{vol.25.1 }{2023}{11}{10.46298/dmtcs.9631}{2022-05-31; 2022-05-31; 2022-11-24}{2023-03-08}
\maketitle
\begin{abstract}
  The chromatic index $\chi'(G)$ of a graph $G$ is the smallest $k$ for which $G$
admits an edge $k$-coloring such that any two adjacent edges have distinct colors. The strong chromatic index $\chi'_s(G)$ of $G$ is the smallest $k$ such that
$G$ has an edge $k$-coloring  with the condition that any two edges at distance at most
2 receive distinct colors. A graph is  1-planar if it can be drawn in the plane so that each edge is
crossed by at most one other edge.

  In this paper, we show that every graph
$G$ with maximum average degree $\bar{d}(G)$ has
 $\chi'_{s}(G)\le (2\bar{d}(G)-1)\chi'(G)$.  As a corollary, we prove that every 1-planar graph $G$
 with maximum degree $\Delta$ has  $\chi'_{\rm s}(G)\le 14\Delta$,
which improves a result, due to Bensmail et al., which says that $\chi'_{\rm s}(G)\le 24\Delta$ if $\Delta\ge 56$. 
\end{abstract}

\section{Introduction}
\label{sec:in}
Only simple graphs are considered in this paper unless otherwise stated.
Let $G$ be a graph with vertex set $V(G)$, edge set $E(G)$, minimum degree $\delta(G)$, and
maximum degree $\Delta(G)$ (for short, $\Delta$), respectively.
A vertex $v$ is called a $k$-{\em vertex} if the degree $d_G(v)$  of $v$ is $k$.
The {\em girth} $g(G)$  of a  graph $G$ is  the length of a shortest cycle in $G$.
The {\em maximum average degree $\bar{d}(G)$} of a graph $G$ is defined as follows:
$$\bar{d}(G)=\max\{\frac {2|E(H)|} {|V(H)|}\ |\ H \subseteq G\}.$$

 A {\em proper edge $k$-coloring} of a graph $G$ is a mapping $\phi: E(G)
\to \{1,2,\ldots ,k\}$ such that $\phi(e)\ne \phi(e')$ for any two
adjacent edges $e$ and $e'$.  The {\em chromatic index} $\chi'(G)$
of $G$ is the smallest $k$ such that $G$ has a proper
edge $k$-coloring.  The coloring $\phi$ is called {\em strong} if any two edges at distance at most two get distinct colors.
Equivalently, each color class  is an induced matching.
 The {\em strong chromatic index}, denoted $\chi'_{\rm s}(G)$, of $G$  is the smallest integer $k$ such that $G$ has a strong edge $k$-coloring.

Strong edge coloring of  graphs was introduced by
Fouquet and Jolivet \cite{fou}.
It holds trivially that $\chi'_s(G)\ge \chi'(G)\ge \Delta$ for any graph $G$.
In 1985, during a seminar in Prague,
Erd\H{o}s and Ne${\rm \breve{s}}$et${\rm \breve{r}}$il
put forward the following conjecture:
\begin{conjecture}
For a simple graph $G$,
\[
\chi'_s(G) \le \left\{
\begin{array}{ll}
  1.25 \Delta^2,
 & \mbox{{\rm if}  $\Delta$ {\rm is even;}}\\
 1.25\Delta^2-0.5\Delta+0.25,
 & \mbox{{\rm if}  $\Delta$ {\rm is odd.}}
\end{array}\right.
\]
\end{conjecture}

Erd\H{o}s and Ne\v{s}et\v{r}il
 provided a construction showing that Conjecture 1 is tight if it were true.
Using probabilistic method,  Molloy and Reed \cite{mol} showed that
$\chi'_s(G)\le 1.998\Delta^2$ for any graph $G$  with sufficiently large $\Delta$.
This result was gradually improved to that
$\chi'_s(G)\le 1.93\Delta^2$ in \cite{bru}, to that $\chi'_s(G)\le 1.835\Delta^2$ in \cite{bon},
and to that $\chi'_s(G)\le 1.772\Delta^2$ in \cite{hur}.  Andersen \cite{and} and independently Hor$\acute{\rm a}$k et al.\,\cite{hor} confirmed Conjecture 1 for graphs with $\Delta=3$. If $\Delta=4$, then   Conjecture 1 asserts that
$\chi'_s(G)\le 20$. However, the  currently best known  upper bound   is  21 for this case, see  \cite{huang}.

A graph $G$ is  $d$-{\em degenerate} if each subgraph  of $G$ contains a vertex of degree at most
$d$. Chang and Narayanan \cite{chan1} showed that $\chi'_{\rm s}(G)\le 10\Delta-10$ for a 2-degenerate graph $G$.
For a general $k$-degenerate graph $G$, it was shown that $\chi'_{\rm s}(G)\le  (4k-2)\Delta-k(2k-1)+1$ in
 \cite{yu}, $\chi'_{\rm s}(G)\le  (4k-1)\Delta-k(2k+1)+1$ in  \cite{dbs},
 and $\chi'_{\rm s}(G)\le  (4k-2)\Delta- 2k^2+1$ in \cite{wan}.

Suppose that $G$ is a planar graph. Faudree et al.\,\cite{fau} first gave an elegant  proof for the result that $\chi'_{\rm s}(G)\le 4\Delta+4$,
and constructed a class of planar graphs $G$ with $\Delta\ge 2$ such that $\chi'_{\rm s}(G)=4\Delta-4$.
For the class of special planar graphs, some better results have been obtained.
It was shown in  \cite{hud}  that  $\chi'_{\rm s}(G)\le 3\Delta$ if $g(G)\ge 7$,
and in \cite{bh} that $\chi'_{\rm s}(G)\le 3\Delta+1$ if $g(G)\ge 6$.
Kostochka et al.\,\cite{kos} showed that if $\Delta=3$ then $\chi'_{\rm s}(G)\le 9$.
Hocquard et al.\,\cite{hoc} showed that every outerplanar graph $G$ with $\Delta\ge 3$ has $\chi'_{\rm s}(G)\le 3\Delta-3$.
Wang et al.\,\cite{yiqiao} showed that every $K_4$-minor-free graph $G$ with $\Delta\ge 3$ has $\chi'_{\rm s}(G)\le 3\Delta-2$.
Moreover, all upper bounds $9, 3\Delta-3,3\Delta-2$ given in the above results are best possible.

A $1$-{\em planar graph} is a graph that can be drawn in the plane such that each edge crosses at most
one other edge. A number of interesting results about structures and parameters of 1-planar graphs
have been obtained in recent years. Fabrici and Madaras \cite{fa} proved that every 1-planar graph $G$ has
$|E(G)|\le 4|V(G)|-8$, which implies that $\delta(G)\le 7$, and constructed   7-regular 1-planar graphs.
Borodin \cite{borodin} showed  that every 1-planar graph is vertex 6-colorable.
Wang and Lih \cite{wang} proved that the vertex-face total graph of a plane graph, which is a class of special 1-planar graphs, is vertex 7-choosable.  Zhang and Wu \cite{zhang} studied the edge coloring of 1-planar graphs and showed that
every 1-planar graph $G$ with $\Delta\ge 10$ satisfies $\chi'(G)=\Delta$.

Bensmail et al.\,\cite{ben} investigated the strong edge coloring of 1-planar graphs and
proved that every 1-planar graph $G$ has $\chi'_{\rm s}(G)\le  \max\{18\Delta+330, 24\Delta-6\}$.
This implies that if $\Delta\ge 56$, then $\chi'_{\rm s}(G)\le 24\Delta-6$.
In this paper we will improve this result by showing that every 1-planar graph $G$ has $\chi'_{\rm s}(G)\le  14\Delta$.
To obtain this result, we establish a connection between the strong chromatic index and maximum average degree of a graph.
More precisely, we will show that $\chi'_{\rm s}(G)\le (2\bar{d}(G)-1)(\Delta+1)$ for any simple graph $G$.

\section{Preliminary}

In this section, we  summarize  some known results, which will be used later.

A  {\em proper $k$-coloring} of a graph $G$ is a mapping $\phi: V(G)\to \{1,2,\ldots,k\}$
such that $\phi(u)\ne \phi(v)$ for any two adjacent vertices $u$ and $v$.
The {\em chromatic number}, denoted $\chi(G)$, of $G$ is the least $k$ such that $G$ has a proper $k$-coloring.

Using a greedy algorithm, the following conclusion holds automatically.

\begin{lemma}\label{vertex-cloring-1}
If $G$ is a $d$-degenerate graph, then $\chi(G)\le d+1$.
\end{lemma}

As stated before, Borodin \cite{borodin} showed the following sharp result:

\begin{theorem}\label{vertex-coloring-2} {\rm (\cite{borodin})}
 Every $1$-planar graph $G$ has $\chi(G)\le 6$.
\end{theorem}

Given a graph $G$, it is trivial that $\chi'(G)\ge \Delta$. On the other hand,
the celebrated Vizing Theorem \cite{vizing} asserts:

\begin{theorem}\label{edge-coloring-1} {\rm (\cite{vizing})}
 Every simple graph $G$ has $ \chi'(G)\le \Delta+1$.
\end{theorem}

A simple graph $G$ is of {\em Class I} if
$\chi'(G)=\Delta$, and of {\em Class II} if $\chi'(G)=\Delta+1$.
As early as in 1916, K$\ddot{\rm o}$nig \cite{konig} showed that  bipartite graphs are of Class I.

\begin{theorem}\label{konig} {\rm (\cite{konig})}
 If $G$ is a bipartite graph, then  $\chi'(G)=\Delta$.
\end{theorem}

  Sanders and Zhao \cite{sander}, and Zhang \cite{zlm} independently,  showed that
  planar graphs with maximum degree at least seven are of Class I.

\begin{theorem}\label{edge-coloring-2}{\rm (\cite{sander,zlm})}
Every planar graph $G$ with $\Delta\ge 7$ has $\chi'(G)=\Delta$.
\end{theorem}

Another result, due to Zhang and Wu \cite{zhang}, claims that 1-planar graphs with maximum degree at least ten are of Class I.

\begin{theorem}\label{edge-coloring-3} {\rm (\cite{zhang})}
Every $1$-planar graph $G$ with $\Delta\ge 10$ has $\chi'(G)=\Delta$.
\end{theorem}

 Zhou \cite{zhou} observed an interesting relation between the degeneracy and chromatic index of a graph.

\begin{theorem}\label{zhou} {\rm (\cite{zhou})}
If $G$ is a $k$-degenerate graph with $\Delta\ge 2k$, then  $\chi'(G)=\Delta$.
\end{theorem}

\section{Contracting matchings in a graph}

Let $G$ be a simple graph.
 An edge $e$ of $G$ is said to be {\em contracted} if it is deleted and its
end-vertices are identified. An edge subset $M$ of $G$ is called a {\em matching}  if no two edges in $M$ are adjacent in $G$.
Specifically, a matching $M$ is called {\em strong} if no two edges in $M$ are adjacent to a common edge.
This is equivalent to saying that $G[V(M)]=M$.  Thus, a strong matching is also called an {\em induced matching}.
Determining the chromatic index $\chi'(G)$ of a graph $G$ is certainly equivalent to
finding the least  $k$ such that $E(G)$ can be partitioned into $k$ edge-disjoint matchings,
and determining the strong chromatic index $\chi'_{\rm s}(G)$ of a graph $G$ is equivalent to
finding the least $k$ such that $E(G)$ can be partitioned into $k$ edge-disjoint strong matchings.
In what follows, an edge $k$-coloring of  $G$ with the color classes $E_1,E_2,\ldots,E_k$
will be denoted by $(E_1,E_2,\ldots,E_k)$.

Given a graph $G$ and a matching $M$ of $G$, let $G_M$ denote the graph obtained from $G$ by contracting each edge in $M$.
Note that $G_M$  may contain multi-edges, but no loops, even if $G$ is simple.

Let $a\ge 1$ and $b\ge 0$ be integers. A graph $G$ is said to be {\em $(a,b)$-graph} if
every subgraph $G'$ of $G$ (including itself)  has $|E(G')|\le a|V(G')|-b$.

\begin{theorem}
  \label{them1} Let $G$ be a  $(a,b)$-graph with $a\ge 1$ and $b\ge 0$.
Let $M$ be a matching of $G$. Then $G_M$ is a $(2a-1,b)$-graph.
\end{theorem}

\begin{proof}
  \ Let $H$ be any subgraph of $G_M$.  Assume that $V(H)=V_1\cup V_2$, where $V_1$ is the set of vertices in $G_M$
which are formed by contracting some edges in $M$, and $V_2=V(H)\setminus V_1$, say,
  $V_1=\{x_1,x_2,\ldots,x_{n_1}\}$,  and $V_2=\{y_1,y_2,\ldots,y_{n_2}\}$. Then
  $|V(H)|=n_1+n_2$. Splitting each vertex $x_i\in V_1$ into two vertices $u_i$ and $v_i$ and
restoring corresponding incident edges for $u_i$ and $v_i$ in $G$,   we get a subgraph $G'$
of $G$ with
$$V(G')=\{u_1,u_2,\ldots,u_{n_1}; v_1,v_2,\ldots,v_{n_1};y_1,y_2,\ldots,y_{n_2}\}$$
and
$$E(G')=E(H)\cup M',$$
 \noindent where
$$M'=\{u_1v_1,u_2v_2,\ldots,u_{n_1}v_{n_1}\} \subseteq M.$$
It is easy to compute that $|V(G')|=2n_1+n_2$ and $|E(G')|=|E(H)|+n_1$.
By the assumption,  $|E(G')|\le  a|V(G')|-b$.
Since $a\ge 1$, we have $2a-1\ge a$.
Consequently,
  \begin{eqnarray*}
  |E(H)|&=& |E(G')|-n_1\\
  &\le&   a|V(G')|-b-n_1\\
  &=& a(2n_1+n_2)-b-n_1\\
   &=& (2a-1)n_1+an_2-b\\
  &\le& (2a-1)(n_1+n_2)-b\\
  &=& (2a-1)|V(H)|-b.
\end{eqnarray*}

This shows that $G_M$ is a $(2a-1,b)$-graph.
\end{proof}

\begin{corollary}\label{coro-4a-3}
Let $G$ be a  $(a,b)$-graph with $a,b\ge 1$.
Let $M$ be a matching of $G$.  Then $G_M$ is $(4a-3)$-degenerate.
\end{corollary}

\begin{proof}
\  It suffices to verify  that $\delta(H)\le 4a-3$  for any $H\subseteq G_M$.
Suppose to the contrary that $\delta(H)\ge 4a-2$. Since $b\ge 1$, Theorem \ref{them1} and the  Handshaking Theorem imply that
  $(4a-2)|V(H)|\le \delta(H)|V(H)|\le \sum\limits_{v\in V(H)}d_H(v)=2|E(H)|\le  2((2a-1)|V(H)|-b) = (4a-2)  |V(H)|-2b<(4a-2)|V(H)|$.
This leads to a contradiction.
\end{proof}

Similarly, we obtain the following consequence:

\begin{corollary}\label{coro-4a-2}
Let $G$ be a  $(a,0)$-graph with $a\ge 1$.
Let $M$ be a matching of $G$.  Then $G_M$ is $(4a-2)$-degenerate.
\end{corollary}

A matching $M$ of a graph $G$ is said to be {\em partitioned} into $q$ strong matchings of
$G$ if $M=M_1\cup M_2\cup \cdots \cup M_q$ and $M_i\cap M_j=\emptyset$ for $i\ne j$ such that each $M_i$ is a strong matching of $G$.
Let $\rho_G(M)$ denote the least $q$ such that $M$ is partitioned into $q$ strong matchings.
By definition, $1\le \rho_G(M)\le |M|$.

The following result is highly inspired from a result of \cite{fau} on the strong chromatic index of planar graphs.
For the sake of completeness, we here give the detailed proof.

\begin{lemma}\label{lem11}\  Let $G$ be a graph and $M$ be a matching of $G$.
Then $\rho_G(M)\le \chi(G_M)$.
\end{lemma}

\begin{proof}
\ Let $V(G_M)=S_1\cup S_2$, where $S_1$ is the set of vertices in $G_M$ formed from $G$
by contracting edges in $M$ and $S_2=V(G)\setminus V(M)$. Set $k=\chi(G_M)$.
Then $G_M$ admits a proper $k$-coloring $\phi: V(G_M)\to \{1,2,\ldots,k\}$.
For $1\le i\le k$, let $V_i$ denote the set of vertices in $G_M$  with the color $i$.
In $G$, for $1\le i\le k$, let

\ \ \ \ \ \ \  $E^*_i=\{e\in M\,|\,e\ {\rm is\ contracted\ to\ some\ vertex}\ v_e\in S_1\ {\rm with}\ \phi(v_e)=i\}.$

Let $e_1,e_2\in E^*_i$ be any two edges. Since  $e_1,e_2\in M$,  $e_1$ and $e_2$ are not adjacent in $G$.
We claim that no edge $e\in E(G)$  is simultaneously  adjacent to both
$e_1$ and $e_2$.  Assume to the contrary, there exists   $e=xy\in E(G)$  adjacent to
  $e_1$ and $e_2$. Without loss of generality, we may suppose that $e_1=xx'$ and $e_2=yy'$.
Let $v_{e_1}$ and $v_{e_2}$ denote the corresponding vertices of $e_1$  and $e_2$  in $S_1$, respectively.
Indeed, $x$ is $v_{e_1}$, and $y$ is $v_{e_2}$.
Since $xy\notin M$, it follows that $xy\in E(G_M)$ and thus $x$ is adjacent to $y$ in $G_M$.
By the definition of $\phi$, $\phi(x)\ne \phi(y)$. Let $\phi(x)=p$ and $\phi(y)=q$.
Then $e_1\in E^*_p$ and $e_2\in E^*_q$ with $p\ne q$, which contradicts the assumption that $e_1,e_2\in E^*_i$. So, each of $E_1^*,E^*_2,\ldots,E^*_k$
is a strong matching of $G$. This confirms that $\rho_G(M)\le k=\chi(G_M)$.
\end{proof}

\section{Strong chromatic index}

In this section, we will discuss the strong edge coloring of some graphs by using the previous preliminary  results.

\subsection{An upper bound}

We first establish an upper bound of strong chromatic index for a general graph $G$,
which reveals a relation between the strong chromatic index, chromatic index and maximum average degree of $G$.

\begin{lemma}\label{average}  Let $H$ be a subgraph of a graph $G$.
Then $|E(H)|\le \frac 12\bar{d}(G)|V(H)|$.
\end{lemma}

\begin{proof}\ For any subgraph $H\subseteq G$, it follows from the definition of $\bar{d}(G)$ that
$\frac {2|E(H)|}{|V(H)|}\le \bar{d}(G)$.   Consequently, $|E(H)|\le \frac 12\bar{d}(G)|V(H)|$.
\end{proof}

\begin{theorem}\label{them3}
Every graph $G$ has $\chi'_{\rm s}(G)\le (2\bar{d}(G)-1)\chi'(G)$.
\end{theorem}

\begin{proof}
 Let $k=\chi'(G)$. Then $G$ has an edge $k$-coloring $(E_1,E_2,\ldots,E_k)$, where each $E_i$ is a matching of $G$.
Let $G_i$ be the graph obtained from $G$ by contracting each of edges in $E_i$.
By Lemma \ref{average} and Corollary \ref{coro-4a-2}, $G_i$ is $(2\bar{d}(G)-2)$-degenerate.
By Lemma \ref{vertex-cloring-1}, $\chi(G_i)\le 2\bar{d}(G)-1$.
By Lemma \ref{lem11},  $\chi'_{\rm s}(G)\le (2\bar{d}(G)-1)k=(2\bar{d}(G)-1)\chi'(G)$.
\end{proof}

By Theorems  \ref{edge-coloring-1}, \ref{konig} and \ref{them3}, the following two corollaries hold automatically.

\begin{corollary}\label{coro-mav-1}
Every graph $G$ has $\chi'_{\rm s}(G)\le (2\bar{d}(G)-1)(\Delta+1)$.
\end{corollary}

\begin{corollary}\label{coro-mav-2}
Every bipartite graph $G$ has $\chi'_{\rm s}(G)\le (2\bar{d}(G)-1)\Delta$.
\end{corollary}

\begin{corollary}\label{coro-mav-3}
If $G$ is a graph with $\Delta\ge 2\bar{d}(G)$, then $\chi'_{\rm s}(G)\le (2\bar{d}(G)-1)\Delta$.
\end{corollary}

\begin{proof}
 \ Since $G$ is $\bar{d}(G)$-degenerate and $\Delta\ge 2\bar{d}(G)$,
Theorem \ref{zhou} asserts that $\chi'(G)=\Delta$. By Theorem \ref{them3},
$\chi'_{\rm s}(G)\le (2\bar{d}(G)-1)\Delta$.
\end{proof}

\subsection{1-planar graphs}

Recently, Liu et al.\,\cite{liu} investigated the existence of  light edges in a 1-planar graph  with minimum degree at least three.
For our purpose, we here list one of their results as follows:

\begin{theorem}\label{7-7} {\rm (\cite{liu})}
Every $1$-planar graph $G$ with $\delta(G)=7$ contains two adjacent $7$-vertices.
\end{theorem}

With  a greedy coloring procedure, it can be constructively shown that the strong chromatic index of
a simple graph $G$ is at most $2\Delta(\Delta-1)+1$.

\begin{theorem}\label{1-planar}\ If $G$ is a $1$-planar graph, then $\chi'_{\rm s}(G)\le 14\Delta$.
 \end{theorem}

\begin{proof}
\  The proof is split into the following cases, depending on the size of $\Delta$.

\begin{enumerate}[{Case }1:]
\item \ $\Delta\le 7$.

It is easy to check that $2\Delta(\Delta-1)+1\le 14\Delta$ and henceforth the result follows.

\item \ $\Delta=8$.

Since $G$ is $7$-degenerate, it follows from the result of \cite{wan} that
$\chi'_{\rm s}(G)\le (4\times 7-2)\Delta-2\times 7^2+1=26\Delta-97=111<112=14\Delta$.

\item \ $\Delta=9$.

The proof is given by induction on the number of edges in $G$.
If $|E(G)|\le 14\Delta=126$, the result holds trivially, since we may color all edges of $G$ with distinct colors.
Let $G$ be a 1-planar graph with $\Delta=9$ and $|E(G)|> 126$.
Without loss of generality,  assume that $G$ is connected, hence $\delta(G)\ge 1$.
We have to consider two subcases as follows.

\begin{enumerate}[{Case }3.1:]
\item\ $\delta(G)\le 6$.

Let $u\in V(G)$ with $d_G(u)=\delta(G)\ge 1$. Let $u_0,u_1,\ldots,u_{s-1}$ denote the neighbors of $u$ in a cyclic order, where
$1\le s=\delta(G)\le 6$. For $0\le i\le s-1$, let $x_i^1,x_i^2,\ldots,x_i^{p_i}$ denote the neighbors of $u_i$ other than $u$.
Consider the graph $H=G-u$. Then $H$ is a 1-planar graph with $\Delta(H)\le 9$ and $|E(H)|<|E(G)|$. By the induction hypothesis or Cases 1 and 2,
$H$ admits a strong edge coloring $\phi$ using the color set $C=\{1,2,\ldots,126\}$.
For a vertex $v\in V(H)$, let $C(v)$ denote the set of colors assigned to the edges incident with $v$.
For $i=0,1,\ldots,s-1$, define a list $L(uu_i)$ of available colors for the edge $uu_i$ as follows:
$$L(uu_i)=C- \bigcup\limits_{0\le j\le s-1; \ j\ne i} C(u_j) -  \bigcup\limits_{1\le t\le p_i} C(x_i^{t}).$$
It is easy to calculate that
\begin{eqnarray*}
|L(uu_i)|    &\ge & |C|- |\bigcup\limits_{0\le j\le s-1; \ j\ne i} C(u_j)|-|\bigcup\limits_{1\le t\le p_i} C(x_i^{t})|\\
  &\ge&   126-(s-1)(\Delta-1)-(\Delta-1)\Delta\\
   &\ge& 126-(6-1)\times (9-1)-(9-1)\times 9\\
   &=&14.
\end{eqnarray*}

Based on $\phi$, we color $uu_0$ with a color $a_0\in L(uu_0)$,
$uu_1$ with a color $a_1\in L(uu_1)\setminus \{a_0\}$, $\cdots,$  $uu_{s-1}$ with a color $a_{s-1}\in L(uu_{s-1})\setminus
\{a_0,a_1,\ldots,a_{s-2}\}$. It is easy to testify that
$\phi$ is extended to whole graph $G$.

\item \ $\delta(G)=7$.

By Theorem \ref{7-7}, $G$ contains two adjacent 7-vertices.
Let $u$ be a 7-vertex of $G$ with neighbors $u_0,u_1,\ldots,u_6$ such that
$d_G(u_0)=7$ and $d_G(u_i)\le 9$ for $i=1,2,\ldots,6$.
Similarly to Case 3.1, for $0\le i\le 6$, let $x_i^1,x_i^2,\ldots,x_i^{p_i}$ denote the neighbors of $u_i$ other than $u$.
Note that $p_0=6$ and $p_i\le 8$ for $i\ge 1$.
Let $H=G-u$, which has a strong edge coloring $\phi$ using the color set $C=\{1,2,\ldots,126\}$, by the induction hypothesis or Cases 1 and 2.
For each $0\le i\le 6$, we define similarly a list $L(uu_i)$ of available colors.
It is easy to check that
 \begin{eqnarray*}
|L(uu_0)|    &\ge & |C|- |\bigcup\limits_{1\le j\le 6} C(u_j)|-|\bigcup\limits_{1\le t\le 6} C(x_0^{t})|\\
  &\ge&   126-6 (\Delta-1)-6\Delta\\
   &=&24.
\end{eqnarray*}
For $1\le i\le 6$,
  \begin{eqnarray*}
|L(uu_i)|    &\ge & |C|-  |C(u_0)|-   |\bigcup\limits_{1\le j\le 6;\ j\ne i} C(u_j)|-|\bigcup\limits_{1\le t\le {p_i}} C(x_i^{t})|\\
  &\ge&   126-6-5(\Delta-1)-8\Delta\\
   &=&8.
\end{eqnarray*}
Based on $\phi$, we color $uu_0$ with a color $a_0\in L(uu_0)$,
$uu_1$ with a color $a_1\in L(uu_1)\setminus \{a_0\}$, $\cdots,$  $uu_{6}$ with a color $a_{6}\in L(uu_{6})\setminus
\{a_0,a_1,\ldots,a_{5}\}$. It is easy to confirm that
$\phi$ is extended to  $G$.
\end{enumerate}
\item \ $\Delta\ge 10$.

By Theorem \ref{edge-coloring-3}, $G$ is of Class I. Let $(E_1,E_2,\ldots,E_{\Delta})$ be an edge
$\Delta$-coloring of $G$, where each $E_i$ is a matching of $G$.
Let $G_i$ be the graph obtained from $G$ by contracting each edge in $E_i$.
Note that each subgraph $H$ of $G$ is 1-planar and therefore $|E(H)|\le 4|V(H)|-8$.
Taking $a=4$ and $b=8$ in Corollary \ref{coro-4a-3}, we deduce that $G_i$ is $13$-degenerate.
By Lemma \ref{vertex-cloring-1}, $\chi(G_i)\le 14$.
Therefore $\chi'_{\rm s}(G)\le 14\Delta$.
\end{enumerate}
\end{proof}

\subsection{Special 1-planar graphs}

Suppose that $G$ is a 1-planar graph which is drawn in the plane so that each edge is
crossed by at most one other edge.  Let $E'$ and $E''$ denote the set of non-crossing edges and crossing edges of $G$, respectively. Let $H_1=G[E']$ and $H_2=G[E'']$. That is, $H_1$ and $H_2$ are the subgraphs of $G$ induced by
non-crossing edges and crossing edges, respectively.
\begin{theorem}\label{special}
Let $G$ be a $1$-planar graph. Then $\chi'_{\rm s}(G)\le 6\chi'(H_1)+14\chi'(H_2)$.
\end{theorem}

\begin{proof}
\  Let $k_1=\chi'(H_1)$ and $k_2=\chi'(H_2)$. Then $\chi'(G)\le k_1+k_2$.
Let $(E_1,E_2,\ldots,E_{k_1})$ be an edge $k_1$-coloring of $H_1$,
and $(F_1,F_2,\ldots,F_{k_2})$ be an edge $k_2$-coloring of $H_2$.
Then each of $E_i$'s and $F_j$'s is a matching in $G$. So,
$(E_1,E_2,\ldots,E_{k_1},F_1,F_2,\ldots,F_{k_2})$ is an edge $(k_1+k_2)$-coloring of $G$.
Similarly to the proof of Case 4 in Theorem \ref{1-planar}, every $F_j$ can be partitioned into 14 strong matchings of $G$.
Moreover, for each $1\le i\le k_1$, because $G_{E_i}$ is a 1-planar graph, we derive that  $\chi(G_{E_i})\le 6$
by Theorem \ref{vertex-coloring-2}. By Lemma \ref{lem11}, $E_i$ can be partitioned into 6 strong matchings.
Consequently, $\chi'_{\rm s}(G)\le 6k_1+14k_2$.
\end{proof}

An {\em IC-planar graph} is a  $1$-planar graph  such that two pairs of crossing
edges have no common end-vertices. Equivalently, each vertex of this kind of 1-planar graph is  incident with at most one crossing edge.
It is easy to verify that  every IC-planar graph $G$ has $|E(G)|\le 3.25|V(G)|-6$  and this bound is attainable.
 Kr\'{a}l and Stacho \cite{kral} showed that every IC-planar graph is vertex 5-colorable.
 Yang et al.\,\cite{ywwl} showed that every IC-planar graph is vertex 6-choosable.
Furthermore, Dvo\v{r}\'ak et al.\,\cite{dv} proved
that every graph drawn in the plane so that the distance between every pair of crossings is at
least 15 is 5-choosable.

Using Theorem \ref{special}, we can establish the smaller upper bound for the strong chromatic index
of IC-planar graphs.

\begin{theorem}\label{IC}
Every IC-planar graph $G$ has $\chi'_{\rm s}(G)\le 6\Delta+20$.
\end{theorem}

\begin{proof}
\ If $\Delta\le 5$, then it is easy to obtain  that
$\chi'_{\rm s}(G)\le 2\Delta(\Delta-1)+1\le 6\Delta+20$ and therefore the theorem holds.
So assume that $\Delta\ge 6$.
Let $H_1$ and $H_2$ denote the graphs induced by non-crossing edges and crossing edges of $G$, respectively.
Since no two crossing-edges  of $G$ are adjacent, $H_2$ is a matching of $G$.
Thus, $\chi'(H_2)\le 1$. Note that $H_1$ is a planar graph with $\Delta(H_1)\le \Delta$.
If $\Delta(H_1)\ge 7$, then $\chi'(H_1)=\Delta(H_1)\le \Delta$ by Theorem \ref{edge-coloring-2}.
So, by Theorem \ref{special}, $\chi'_{\rm s}(G)\le 6\chi'(H_1)+14\chi'(H_2)\le 6\Delta+14$.
Otherwise, we have to consider two cases as follows:

$\bullet$\   $\Delta(H_1)=6$. Then $6\le \Delta\le 7$.
By Theorem \ref{edge-coloring-1}, $\chi'(H_1)\le 7$. By Theorem \ref{special},
   $\chi'_{\rm s}(G)\le 6\chi'(H_1)+14\chi'(H_2)\le   6\times 7+14= 56\le 6\Delta+20$.

$\bullet$\   $\Delta(H_1)=5$. Then $\Delta=6$ by the assumption.
By Theorem \ref{edge-coloring-1}, $\chi'(H_1)\le 6$. By Theorem \ref{special},
   $\chi'_{\rm s}(G)\le 6\chi'(H_1)+14\chi'(H_2)\le   6\times 6+14= 50=6\Delta+20$.
\end{proof}

A 1-planar graph $G$ is called {\em optimal} if $|E(G)|=4|V(G)|-8$.
A {\em plane quadrangulation} is a plane graph such that each face of $G$ is of degree 4.
It is not hard to show that a 3-connected plane quadrangulation is a bipartite plane graph with minimum degree 3.
Suzuki \cite{suz} showed that every simple optimal 1-planar graph $G$ can be obtained
from a 3-connected plane quadrangulation by adding a pair of crossing edges to each face of $G$.
So an optimal 1-planar graph is an Eulerian graph, i.e., each vertex is of even degree.
It was shown in \cite{len} that every optimal 1-planar graph $G$ can be edge-partitioned into two planar graphs $G_1$ and $G_2$
such that $\Delta(G_2)\le 4$.

\begin{theorem}\label{optimal}
Every optimal 1-planar graph $G$ has $\chi'_{\rm s}(G)\le 10\Delta+14$.
\end{theorem}

\begin{proof}
\  Let $G$ be an optimal 1-planar graph.
Let $H_1$ and $H_2$ denote the graphs induced by non-crossing edges and crossing edges of $G$, respectively.
Then $G=H_1\cup H_2$, where $H_1$ is a bipartite plane graph. For each vertex $v\in V(G)$, it is easy to see that $d_{H_1}(v)=d_{H_2}(v)=\frac 12 d_G(v)$;
in particular, we have $\Delta(H_1)=\Delta(H_2)=\frac {\Delta}2$.

Since $H_1$ is bipartite, $\chi'(H_1)=\Delta(H_1)=\frac {\Delta}2$ by
Theorem \ref{konig}. By Theorem \ref{edge-coloring-1},
$\chi'(H_2)\le \Delta(H_2)+1=\frac {\Delta}2+1$.
By Theorem \ref{special},
 $\chi'_{\rm s}(G)\le 6\chi'(H_1)+14\chi'(H_2)\le   6 \times \frac {\Delta}2+14\times (\frac {\Delta}2+1)  = 10\Delta+14$.
\end{proof}

\section{Concluding remarks}

In this paper, we show that the strong chromatic index of every 1-planar graph is at most $14\Delta$.
As for the lower bound of strong chromatic index,
Bensmail et al.\,\cite{ben} showed that for each $\Delta\ge 5$, there exist 1-planar graphs with strong chromatic index $6\Delta-12$.
Based on these facts, we put forward the following:

\medskip
\noindent{\bf Question 1.}\ {\em What is the least constant $c_1$ such that every $1$-planar graph $G$ satisfies  $\chi'_{\rm s}(G)\le c_1\Delta\,?$}
\medskip

The foregoing discussion asserts that  $6\le c_1\le 14$. We think that it is very difficult to reduce further the value of $c_1$
by employing the method used in this paper.

This paper also involves the strong edge coloring of some special 1-planar graphs such as IC-planar graphs and optimal 1-planar graphs.
In particular, we show that the strong chromatic index of every IC-planar graph is at most $6\Delta+20$.
For $\Delta\ge 4$, by attaching $\Delta-4$ new pendant vertices to each vertex
of the complete graph $K_5$, we get a graph $H_{\Delta}$.
Since $K_5$ is an IC-planar graph, so is $H_{\Delta}$.
It is easy to inspect that any two edges of $H_{\Delta}$ lie in a path of length 2 or 3. So it follows that $\chi'_{\rm s}(H_{\Delta})=|E(H_{\Delta})|=10+5(\Delta-4)=5\Delta-10$.

\medskip
\noindent{\bf Question 2.}\ {\em What is the least constant $c_2$ such that every IC-planar graph $G$ satisfies  $\chi'_{\rm s}(G)\le c_2\Delta\,?$}
\medskip

Notice that $5\le c_2\le 6$. We conjecture that $c_2=5$.


\end{document}